\documentclass[reqno]{amsproc}

\usepackage{amsfonts}
\usepackage{amsmath}
\usepackage{amssymb}
\usepackage{enumerate}

\newtheorem{theorem}{Theorem}
\newtheorem{lemma}{Lemma}
\newtheorem{corollary}{Corollary}

\theoremstyle{definition}
\newtheorem{definition}{Definition}

\theoremstyle{remark}
\newtheorem{remark}{Remark}

\numberwithin{equation}{section}

\def\be{\begin{equation}}
\def\ee{\end{equation}}
\def\ben{\begin{displaymath}}
\def\een{\end{displaymath}}
\def\baa{\begin{eqnarray}}
\def\eaa{\end{eqnarray}}

\def\ba{\begin{array}}
\def\ea{\end{array}}
\def\la{\label}
\def\p{\partial}

\def\strat{{\mathcal H}_{\gh}([2^{\gh-1}])}
\def\C{{\mathbb C}}
\def\R{{\mathbb R}}
\def\Z{{\mathbb Z}}
\def\Q{{\mathbb Q}}

\def\H{{\mathcal H}}

\def\ph{\widehat{p}}

\def\Oh{\widetilde{\phi}}

\def\Mb{\overline{{\mathcal M}}}

\def\2x2{{\left(\!\!\begin{array}{cc}a&b\\c&d\\\end{array}\!\!\right)}}

\def\f{\frac}
\def\e{\epsilon} 

\def\d{\delta}
\def\a{\alpha}
\def\b{\beta}
\def\s{\sigma}

\def\O{\Omega}

\def\Ch{\widehat{C}}
\def\gh{\widehat{g}}
\def\Bh{{\widehat B}}
\def\l{\lambda}

\def\sh{\widehat{\sigma}}
\def\thetah{\widehat{\theta}}
\def\Eh{\widehat{E}}
\def\Ah{{\widehat{{\mathcal A}}}}
\def\Kh{\widehat{K}}
\def\Wh{\widehat{W}}
\def\thetah{\widehat{\theta}}
\def\Zh{{\widehat{Z}}}
\def\zetah{\widehat{\zeta}}
\def\wO{{\widehat{\Omega}}}

\def\Bt{{\widetilde{B}}}

\def\pt{\lambda_{P}}
\def\at{\tilde{a}}
\def\bt{\tilde{b}}
\def\ut{\tilde{u}}
\def\vt{\tilde{v}}

\def\bu{{\bf  u}}

\def\dim{{\rm dim}}

\def\Bh{{\widehat{B}}}

\def\bu{\pmb{u}}

\def\qb{{\xi}}

\def\tauh{\widehat{\tau}}
\def\taut{\widetilde{\tau}}
\def\gh{\widehat{g}}

\def\Mgn{\mathcal{Q}_{g}}
\def\Mgno{\overline{\mathcal{Q}}_{g}}

\def\Ct{C_{\gh}}
\def\Cb{\overline{{\mathcal C}}}
\def\Qc{{\mathcal Q}}
\def\Ct{{\widetilde C}}

\begin{document}

\title[Tau function and the Prym class]{Tau function and the Prym class}


\author{D.~Korotkin}
\address{Department of Mathematics and Statistics, Concordia University,
1455 de Maisonneuve West, Montreal, H3G 1M8  Quebec,  Canada}
\email{korotkin@mathstat.concordia.ca}
\thanks{DK was partially supported by NSERC, FQRNT and CURC}

\author{P.~Zograf}
\address{Steklov Mathematical Institute, Fontanka 27, St. Petersburg 191023 Russia, and \newline 
Chebyshev Laboratory, St. Petersburg State University, 14th Line V.O. 29, St. Petersburg 199178 Russia}
\email{zograf@pdmi.ras.ru}
\thanks{PZ was partially supported by the Government of the Russian Federation megagrant 11.G34.31.0026 and by the RFBR grants 11-01-12092-OFI-M-2011 and 11-01-00677-a}


\date{}

\begin{abstract}
We use the formalism of the Bergman tau functions to study the geometry of moduli spaces of holomorphic quadratic differentials on complex algebraic curves. We introduce two natural tau functions and interpret them as holomorphic sections of certain line bundles on the moduli space.
Analyzing the asymptotic behavior of these tau functions near the boundary of the moduli space we get two non-trivial relation in the rational Picard group of the moduli space of quadratic differential.
\end{abstract}

\maketitle

\section{Introduction and statement of results}

Tau functions play an important role in the theory of integrable systems providing canonical generators for commuting flows on the phase space. In some cases tau functions carry interesting geometric information,
like the isomonodromic tau function of the Riemann-Hilbert problem that is relevant in the theory
of Frobenius manifolds \cite{D}. The asymptotic analysis of the Bergman tau function introduced in \cite{IMRN} allowed to express the Hodge
class on the space of admissible covers of the projective line as an explicit linear combination of the boundary 
divisors \cite{Adv}. Then in \cite{KG1} this result was proven by pure algebraic geometry methods (namely, 
Grothendieck-Riemann-Roch theorem) and later in  \cite{KG} it was used to answer a question of Harris-Mumford \cite{HM} about 
the classes of the Hurwitz divisors in the moduli space $\Mb_g$ of stable nodal complex algebraic curves of even genus $g$. 
A version of the Bergman tau function for the moduli space of abelian differentials on algebraic
curves \cite{JDG} appeared to be useful in understanding the relations in the rational Picard group of this space
and was applied to the Kontsevich-Zorich theory of Teichm\"uller flow \cite{MRL}, see also \cite{EKZ}.

Here we continue to develop these ideas further, and apply the Bergman tau functions to study the geometry
of the moduli space $\Mgno$ of holomorphic quadratic differentials on complex algebraic curves of genus
$g$ (throughout the paper we assume that $g\geq 2$). 
The space $\Mgno$ is defined as follows. Consider the universal curve $\pi:\Cb_g\to\Mb_g$, 
and the relative dualizing sheaf $\omega_g$ on $\Cb_g$. We put $\Mgno=\pi_*\omega_g^2$ (i.e. the total space of the direct image of $\omega_g^2$ on $\Mb_g$). There is a natural projection $\Mgno\to\Mb_g$ that is a vector bundle of dimension $3g-3$
(in the orbifold sense). The points of $\Mgno$ are represented by the isomorphism classes of pairs $(C,q)$, where $C$ 
is a stable genus $g$ algebraic curve, and $q$ is a {\em regular} quadratic differential on $C$. The differential $q$ is holomorphic 
on a smooth curve $C$, and may have poles of order up to $2$ at the nodes of $C$ (in the case of poles the residues
of $q$ at the cusp for the two branches of $C$ coincide).
The moduli space $\Mgno$ has an open dense subset $\Mgn$, that consists of equivalence classes of pairs $(C,q)$,
where $C$ is a smooth curve, and $q$ has only simple zeros. We call $\Mgn$ the moduli space of {\em generic}
holomorphic quadratic differentials. 

Denote by $P\Mgno$ the projectivization of $\Mgno$, and put $P\Mgn=\Mgn/\C^*$, where $\C^*$ acts by multiplication on quadratic differentials. The complement $P\Mgno\setminus P\Mgn$ is a union of $[g/2]+2$ divisors that we denote
$D_{deg}, D_0,\ldots, D_{[g/2]}$, where $D_{deg}$ is the divisor of {\em degenerate} quadratic differentials 
(i.e. having at least one zero of multiplicity larger than 1), and $D_i\; (i=0,\ldots, [g/2])$ are the inverse images of the 
corresponding boundary divisors in $\Mb_g$ with respect to the projection $p:P\Mgno\to\Mb_g$. Let $L\to P\Mgno$ be the tautological line bundle associated with the projectivization $\Mgno\to P\Mgno$,
and put $\psi=c_1(L)\in {\rm Pic}(P\Mgno)\otimes\Q$. Denote by $\lambda$ the pullback of the Hodge class $\lambda_1=\det\pi_*\omega_g$ on $\Mb_g$. 
We also put 
$\delta_{i}=[D_{i}]$ for $i\neq 1$,  $\delta_1=\f{1}{2}[D_1]$ and $\delta_{deg}=[D_{deg}]$ in ${\rm Pic}(P\Mgno)\otimes\Q$.
The following fact is standard (cf., e.g, \cite{MRL}, Lemma 1):
\begin{lemma}\label{pic}
The rational Picard group ${\rm Pic}(P\Mgno)\otimes\Q$ is freely generated over $\Q$ by
$\psi, \lambda, \delta_0,\dots,\delta_{[g/2]}$.
\end{lemma}

To each pair $(C,q)$ one can canonically associate a twofold branched cover $f:\Ch\to C$ and an abelian differential $v$ on $\Ch$, where
$\Ch=\{(x,v(x))\,|\;x\in C,\;v(x)\in T_x^*C,\;v(x)^2=q(x)\}$. For generic $(C,q)\in\Mgn$
the curve $\Ch$ is smooth of genus $\gh=4g-3$ and $v$ is holomorphic.
The covering $f$ is invariant with respect to the canonical involution $(x,v(x))\mapsto (x,-v(x))$ on $\Ch$ that we denote by $\mu$.

Consider the natural map $\ph:P\Mgno\to \Mb_{\gh},\quad (C,q)\mapsto \Ch$. This map induces a vector bundle $\ph^*\Lambda^1_{\gh}\to \Mgno$
of dimension $\gh=4g-3$, where $\Lambda^1_{\gh}\to \Mb_{\gh}$ is the Hodge vector bundle. The involution $\mu$ on
$\Ch$ naturally induces an involution on the vector bundle $\ph^*\Lambda^1_{\gh}$. Hence we have a decomposition 
\be
\ph^*\Lambda^1_{\gh}=\Lambda_+\oplus \Lambda_-
\ee 
where $\Lambda_+$ (resp. $ \Lambda_-$) is the eigenbundle corresponding to the eigenvalue $+1$ (resp. $-1$).
Clearly, $\Lambda_+=p^*\pi_*\omega_g$ is the pullback of the Hodge bundle on $\Mb_g$. We call  $\Lambda_-$ the {\em Prym bundle} (since its fibers are the spaces of Prym differentials on $\Ch$), and we call $\pt=c_1(\Lambda_-)\in {\rm Pic}(P\Mgno)\otimes\Q$ the {\em Prym
class} (we will often denote the line bundles and their classes in the Picard group by the same symbols).

 The main result of the paper is
\begin{theorem}\la{theoint}
The Hodge class $\lambda$ and the Prym class $\pt$  can be expressed in terms of the tautological class $\psi$ and
the classes $\delta_{deg}$, $\delta_i ,\; i=0,\ldots, [g/2],$ as follows:
\begin{align}
&\lambda=\f{5(g-1)}{36}\,\psi +\f{1}{72}\, \delta_{deg}+\f{1}{12}\sum_{i=0}^{[g/2]}\delta_i\;,\la{H}\\
&\pt=\f{11(g-1)}{36}\,\psi +\f{13}{72}\, \delta_{deg}+\f{1}{12}\sum_{i=0}^{[g/2]}\delta_i\;.
\la{PTMbar}
\end{align}
\end{theorem}

The proof is analytic and is inspired by the formalism of the 
Bergman tau functions from the theory of integrable systems. In particular, this theorem immediately implies that
\be
\pt - 13 \l= -\sum_{i=0}^{[g/2]}\delta_i -\f{3g-3}{2} \psi \;.
\la{pt10}
\ee
On the other hand, Mumford \cite{M} proved the following relation between the classes 
 $\l_n=c_1(\pi_*\omega_g^n)\in{\rm Pic}(\Mb_g)\otimes\Q$: 
\be
\l_n=(6n^2-6n+1)\lambda_1 - \f{n(n-1)}{2}\sum_{i=0}^{[g/2]}\delta_i\;.
\label{M}
\ee
It is easy to see that $\l= p^*\l_1$ and 
 $\pt= p^*\l_2-\f{3g-3}{2} \psi$. Therefore, Mumford's formula (\ref{M}) for $n=2$ 
is a direct consequence of  (\ref{pt10}). We notice, however, that Theorem \ref{theoint} provides more information about the classes $\l$ and $\l_P$ than  (\ref{M}) does.

\bigskip
The paper is organized as follows. In Section \ref{doubcov} we introduce a two-fold canonical cover corresponding to 
a pair consisting of a complex algebraic curve and a holomorphic quadratic differential on it. The main objective of
this section is to discuss the action of the covering involution on (co)homology of the cover and the associated matrix of $b$-periods. In Section \ref{varform} we introduce homological coordinates on the moduli space
of quadratic differentials and prove variational formulas of Ahlfors-Rauch type for the matrix of $b$-periods, normalized 
abelian differentials etc. with respect to these coordinates. In Section \ref{secdeftau} we define two tau functions corresponding to the eigenvalues $\pm 1$ of the covering map, discuss their basic properties and interpret them as holomorphic sections of line bundles on the moduli space of quadratic differentials. In Section \ref{secasymp} we study the asymptotic behavior of the
tau functions near the boundary of the moduli space and use it to express the Hodge and Prym classes via the classes of 
boundary divisors and the tautological class.

\section{Geometry of the double cover}
\la{doubcov}

Let $f:\Ch\to C$ be the double cover defined by the holomorphic quadratic differential $q$ with simple zeroes on a smooth curve $C$, and let $\mu:\Ch\to\Ch$ be the corresponding involution. Furthermore, let $\{x_1,\dots,x_{4g-4}\}\subset\Ch$ be the fixpoints of $\mu$ (or, in other words, the zeros of the abelian differential $v=\sqrt{q}$ on $\Ch$). By $\mu_*$ (resp. $\mu^*$) we denote the involution induced by $\mu$ in homology (resp. in cohomology) of $\Ch$.
As we mentioned above, the space $\Lambda^1_{\Ch}$ of abelian differentials on $\Ch$ splits into two eigenspaces $\Lambda_+$ and $\Lambda_-$ of complex dimension $g$ and $3g-3$ respectively that correspond to the eigenvalues $\pm 1$ of $\mu^*$. We have a similar decomposition in the real homology of $\Ch$: $H_1(\Ch,\R)=H_+\oplus H_-$, where $\dim\, H_+=2g$, $\dim\, H_-=6g-6$.

Following \cite{Fay73}, we pick $8g-6$ smooth 1-cycles on $\Ch\setminus\{x_1,\dots,x_{4g-4}\}$
\be
\{a_j,  a_j^*, \at_k,b_j,  b_j^*,\bt_k\}\,,\quad j=1,\dots,g,\;k=1,\dots,2g-3,
\label{mainbasis}
\ee
in such a way that
$$\mu_* a_j = a_j^*,\quad\mu_* b_j = b_j^*,
\quad\mu_* \at_k + \at_k= \mu_* \bt_k + \bt_k  = 0,$$
and the intersection matrix is
$$\left(\ba{cc} 0 & I_{4g-3}\\- I_{4g-3} & 0 \ea\right)$$
(here $I_n$ denotes the $n\times n$ identity matrix).

\begin{remark}
To avoid complicated notation, we will use the same symbols for the cycles (\ref{mainbasis}), their homology classes
in $H_1(\Ch\setminus\{x_1,\dots,x_{4g-4}\})$, and their pushforwards in $H_1(\Ch)$ and the relative homology group $H_1(\Ch,\{x_1,\dots,x_{4g-4}\})$. In particular, the images of (\ref{mainbasis}) give rise to a canonical basis in $H_1(\Ch)$.
\end{remark}

The basis of normalized abelian differentials on $\Ch$ associated with (\ref{mainbasis}) we denote by 
$\{u_j, u^*_j ,\ut_k \}$, so that the action of $\mu^*$ on  
$\Lambda^1_{\Ch}$ is given by the matrix
\be
M=  \left(\ba{ccc} 0 & I_g &0 \\
                     I_g  & 0 & 0 \\
                  0 & 0 & -I_{2g-3} \ea\right) 
\la{Smu}
\ee
The differentials $u_j^+=u_j+u_j^*,\; j=1,\dots,g$,
provide a basis in the space $\Lambda_+$, whereas a basis in $\Lambda_-$ is given by  $3g-3$ Prym differentials $u_l^-$, where
\be
u_l^-=\begin{cases}u_{l}-u_{l}^*,\quad l=1,\dots, 3\;,\\
\ut_{l-g}, \qquad\; l=g+1,\dots, 3g-3\;.\end{cases}
\la{Prym}
\ee

We also introduce the bases in the spaces $H_+$ and $H_-$.
The classes
\be
\a_j^+ = \f{1}{2}(a_j+ a_j^*)\;,\quad
\b_j^+ = b_j+ b_j^*\;,\quad j=1,\dots,g,
\la{abp}
\ee
form a symplectic basis in $H_+$, whereas the  classes
\begin{align}
&\a_l^- = \f{1}{2}(a_{l}- a_{l}^*) ,\quad\;
\b_l^- = b_{l}- b_{l}^*,\quad l=1,\dots,g\;,\nonumber\\
&\a^-_{l}=\at_{l-g},\qquad\qquad \b^-_{l}=\bt_{l-g}\;,\qquad\, l=g+1,\dots,3g-3
\la{abm}
\end{align}
form a symplectic basis in $H_-$.
The basis $\{a_j^+,\a_l^-,\b_j^+,\b_l^-\}$, $j=1,\dots,g$, $l=1,\dots,3g-3$, is related to the canonical basis (\ref{mainbasis}) by means of a (non-integer) symplectic matrix
\be
S=  \left(\ba{cc} T & 0 \\
                    0 & (T^t)^{-1}
 \ea\right) 
\la{mainmat}
\ee
with
\be
T=  \left(\ba{ccc} I_g & I_g& 0     \\
                     I_g & -I_g  & 0 \\
                    0& 0 & I_{2g-3}  \ea\right)\;.
\la{MaMb}
\ee
The differentials  $\{u_j^+,u_l^-\}$ are  normalized with respect to the classes $\{\a_j^+, \a_l^-\}$ in the sense that
$$\int_{\a_j^\pm} u_i^\pm=\delta_{ij}\;.$$
The corresponding matrices of $\beta$-periods $\Omega_+$ and $\Omega_-$ are given by
\be
(\Omega_+)_{ij}=\int_{\a_i^+} u_j^+\;,\quad i,j=1,\dots, g,\\
\la{sigp}
\ee
\be
(\Omega_-)_{ij}=\int_{\a_i^-} u_j^-\;, \quad i,j=1,\dots, 3g-3.
\la{sigm}
\ee
 The matrix $\widehat{\Omega}$ of $b$-periods of $\{u_j, \ut_k, u^*_j \}$ with respect to the homology basis (\ref{mainbasis}) on $\Ch$ is related to $\Omega^+$ and $\Omega^-$ by the formula
\be
\widehat{\Omega}=T^{-1}
 \left(\ba{cc} \Omega_+ & 0\\
                0  & \Omega_-  \ea\right) (T^t)^{-1}
\la{shs}
\ee

We proceed with bidifferentials and projective connections
on the double covers.
Let  $\Bh(x,y)$ denote the canonical (Bergman) bidifferential  on $\Ch\times\Ch$ associated with the homology basis (\ref{mainbasis}); $\Bh(x,y)$ is symmetric, has the second order pole on the diagonal $x=y$ with biresidue 1
 and satisfies
\be
\int_{a_j}\Bh(\,\cdot\,,y)=\int_{a_j^*}\Bh(\,\cdot\,,y)=\int_{\at_i}\Bh(\,\cdot\,,y)=0\;
\la{normbh}\ee
(here and below the dot stands for the integration variable).
Equivalently, (\ref{normbh}) can be written in terms of the cycles $\a_i^+$ and $\a_i^-$ defined by
(\ref{abp}) and (\ref{abm}):
\be
\int_{\a_j^+}\Bh(\,\cdot\,,y)=\int_{\a_l^-}\Bh(\,\cdot\,,y)=0\;.
\la{normbh1}\ee
for $j=1,\dots,g$ and $l=1,\dots,3g-3$.
Clearly, $\Bh(x,y)$ is invariant under the involution $(\mu_x,\mu_y): (x,y)\mapsto (\mu(x),\mu(y))$ on $\Ch\times\Ch$.

Now put
\be
B_+(x,y)=\Bh(x,y)+\mu_y^*\Bh(x,y)\,, \quad B_-(x,y)=\Bh(x,y)-\mu_y^*\Bh(x,y)\,,
\la{defBpm}
\ee 
(the subscript $y$ at $\mu^*$ means that we take the pullback with respect to the involution on the second factor in $\Ch\times\Ch$).
While $B_+(x,y)$ is just the pullback of the canonical bidifferential $B(x,y)$ on $C\times C$ (normalized relative to the classes $p_* a_j$), the bidifferential  $B_-$ looks like a new object that we call the {\it Prym bidifferential}.

\begin{lemma}
The  bidifferentials  $B_+$ and $B_-$ have the following properties:
\begin{enumerate}[(i)]
\item
\be
\mu_x^*\mu_y^*B_\pm(x,y)=B_\pm(x,y),
\la{muxy}
\ee
\item
\be
B_\pm(y,x)=B_\pm(x,y),
\la{sympm}
\ee
\item
\begin{align}
&\mu_x^*B_+(x,y)=\mu_y^*B_+(x,y)=B_+(x,y), \nonumber\\ &\mu_x^*B_-(x,y)=\mu_y^*B_-(x,y)=-B_-(x,y),
\la{Bpm}
\end{align}
\item
for any $s_+\in H_+,\;s_-\in H_-$
\be
\int_{s_-} B_+(\,\cdot\,,y)=0, \quad\quad \int_{s_+} B_-(\,\cdot\,,y)=0,
\la{BpH}
\ee
\item
\be
\int_{\beta^\pm_i} B_\pm(\,\cdot\,,y)= 4\pi \sqrt{-1} u_i^\pm(y)
\la{perBpm}
\ee 
\end{enumerate}
\end{lemma}
\begin{proof} The proof is elementary and immediately follows from the definitions.
\end{proof}
Near the diagonal $x=y$ in $\Ch\times\Ch$ we have
\be
\Bh(x,y)= \f{d\zeta(x)d\zeta(y)}{(\zeta(x)-\zeta(y))^2}+\f{1}{6}S_\Bh(\zeta(x))+\dots
\la{Bhatdiag}
\ee
where $\zeta(x)$ is a local parameter at $x\in\Ch$, and $S_\Bh(x)$ is a projective connection on $\Ch$. For $B_\pm(x,y)$ near the diagonal we have
\be
B_\pm(x,y)= \f{d\zeta(x)d\zeta(y)}{(\zeta(x)-\zeta(y))^2}+\f{1}{6}S_{B_{\pm}}(\zeta(x))
+\dots
\la{Bpmdiag}
\ee
with two  projective connections $S_{B_+}$ and $S_{B_-}$ that are related by
\be
S_{B_\pm}(x)=S_{\Bh}(x)\pm 6 \mu_y^*\Bh(x,y)\left|_{y=x}\right..
\la{defSbpm}
\ee
We will call $S_{B_-}$  the {\it Prym projective connection}.

Now we describe the dependence of bidifferentials and projective connections on the choice of  homology basis.
Let $\sigma$  be a symplectic (i.e., preserving the intersection form) transformation of $H_1(\Ch,\Z)$.  In the canonical basis (\ref{mainbasis}), $\sigma$ is represented by an $Sp(8g-6,\Z)$-matrix
$\left(\ba{cc} A & B\\C  & D  \ea\right)$ with square blocks of size $4g-3$. The canonical  bidifferential $\Bh^\sigma(x,y)$ on $\Ch\times\Ch$ associated with the new basis is
\be
\Bh^\sigma (x,y) = \Bh(x,y) -2\pi \sqrt{-1} \bu(x)^t (C\wO+ D)^{-1} C\, \bu(y),
\la{transB}
\ee
where ${\bu}=\{u_j,\tilde{u_k},u_j^*\}^t$ is the vector of holomorphic abelian differentials on $\Ch$ normalized with respect to $\{a_j,\tilde{a_k},a_j^*\}\,,\;
j=1,\dots,g,\;k=1,\dots, 2g-3.$

If $\sigma$ commutes with the involution $\mu_*$, then $\sigma={\rm diag}(\sigma_+,\sigma_-)$, where $\sigma_+$ (resp. $\sigma_-$) is a symplectic transformation of $H_+$ (resp. $H_-$) with half-integer coefficients (note that the intersection form on $H_1(\Ch,\R)$ is invariant with respect to $\mu_*$). In the bases $\{\a_j^+,\b_j^+\},\; j=1,\dots,g,$ and $\{\a_l^-,\b_l^-\},\;
l=1,\dots, 3g-3,$ these transformations can be written as
\be
\sigma_\pm:=  \left(\ba{cc} A_\pm & B_\pm\\
                C_\pm  & D_\pm  \ea\right).
\la{Qpm}
\ee
The following is immediate:
\begin{lemma} 
Let ${\bu}_+$ (resp. ${\bu}_-$) be the vector of holomorphic abelian differentials $\{u_j^+\}$ in $\Lambda_+$ 
(resp.  $\{u_l^-\}$ in $\Lambda_-$) normalized relative to
the classes $\{\a_j^+\}$ in $H_+$ (resp. $\{\a_j^-\}$ in $H_-$), 
$j=1,\dots,g,\;l=1,\dots, 3g-3$.
Then for $\sigma$ commuting with $\mu_*$ we have
\begin{enumerate}[(i)]
\item
\be
{\bu^\sigma}_\pm = (C_\pm\Omega_\pm+ D_\pm)^{-1} {\bu}_\pm,
\la{transupm}
\ee
\item
\be
\Omega_\pm^\sigma = (C_\pm\Omega_\pm+ D_\pm)^{-1} (A_\pm\Omega_\pm+ B_\pm),
\la{transspm}
\ee
\item
\begin{align}
&B_+^\sigma (x,y) = B_+(x,y) -4\pi \sqrt{-1} \bu_+^t(x) (C_+\Omega_++ D_+)^{-1} C_+ \bu_+(y),\la{transBp}\\
&B_-^\sigma (x,y) = B_-(x,y) -4\pi \sqrt{-1} \bu_-^t(x) (C_-\Omega_-+ D_-)^{-1} C_- \bu_-(y),
\la{transBm}
\end{align}
\item
\be
S_{B_\pm}^\sigma= S_{B_\pm} - 24\pi \sqrt{-1} \bu_\pm^t (C_\pm\Omega_\pm+ D_\pm)^{-1} C_\pm \bu_\pm\;.
\la{transSb}
\ee
\end{enumerate}

\end{lemma}

\section{Variational formulas}
\la{varform}

Let $(C,q)$ represent a point in $\Mgn$ (i.e. $C$ is smooth of genus $g$, and $q$ is holomorphic with $4g-4$ simple zeros). Let $(\Ch, v)$ be the corresponding canonical double cover and  abelian differential. Choose a canonical basis in $H_1(\Ch,\Z)$ of the form (\ref{mainbasis}), and consider the basis $\{\a_i^-,\b_{i}^-\}$ in $H_-$ given by (\ref{abm}).
Put for brevity
\be
s_{2i-1}=\a_i^-,\quad s_{2i}=\b_i^-,\qquad i=1,\dots,3g-3,
\la{si}
\ee
and consider $6g-6$ complex parameters
\be
z_i=\int_{s_i} v,\quad i=1,\dots,6g-6
\la{zi}
\ee
defined in a neighborhood of $(C,q)$ in $\Mgn$.

\begin{lemma}\la{lemmacoord}
The parameters $z_i,\; i=1,\dots,6g-6,$ provide a system of local coordinates on $\Mgn$.
\end{lemma}
\begin{proof}
Since the differential $v$ has a double zero at each of $4g-4$ branch points of the covering $\Ch\to C$,
the pair $(\Ch,v)$ belongs to the stratum $\strat$ of the moduli space of abelian differentials on
algebraic curves of genus $\gh$ with $\gh-1$ zeroes of multiplicity $2$ (cf. \cite{JDG} for details). The dimension of the space $\strat$ is 
$3\gh-2= 12g-11$ and the space $\Mgn$ (of dimension  $6g-6$) forms a subspace of codimension $6g-5$ in
 $\strat$. This subspace can easily be described in terms of homological coordinates on  $\strat$.

Let $(\widetilde{C},\tilde{v})$ represent a point in $\strat$.
The set of homological coordinates on $\strat$ consists of integrals of $\tilde{v}$ along any set of cycles representing a basis in the relative homology group $H_1(\Ct, \{x_1,\dots,x_{\gh-1}\})$; such a basis 
consists of a canonical basis (\ref{mainbasis}) in $H_1(\Ct)$, and of the relative homology classes of $\gh-2=4g-5$ paths $l_2,\dots,l_{4g-4}$ connecting 
$x_1$ with $x_j,\; j=2,\dots,4g-4$,  and not intersecting the cycles (\ref{mainbasis}).
If a pair  $(\Ct,\tilde{v})$ belongs to $\Mgn$, $\Ct$ is invariant under a holomorphic involution $\mu$ and 
$\mu^*\tilde{v}=-\tilde{v}$ (in particular,
each zero $x_i$ of $v$ is invariant under $\mu$). 
Therefore, when $\Ct=\Ch$ and $\tilde{v}=v$, each union
$l_j\cup\mu(l_j)$ is a cycle on $\Ch$ and can be decomposed into a linear combination of the elements of the basis (\ref{mainbasis}).
Clearly, $\int_{l_i}v =-\int_{\mu(l_i)}v$, so that each coordinate $\int_{l_i}v$ on $\Mgn$ becomes a linear combination of $8g-6$ coordinates given by the periods of $v$ along $a$- and $b$-cycles on $\Ch$ with half-integer coefficients.
Moreover, we have  $\int_{a_j^*} v= -\int_{a_j} v$ and  $\int_{b_j^*} v= -\int_{b_j} v$ for each $j=1,\dots,g$, which
gives $2g$ more vanishing linear combinations of homological coordinates on $\Mgn$. 
The remaining $6g-6$ homological coordinates associated with a basis in $H_-$ can obviously serve as local coordinates on $\Mgn$. 
\end{proof}

This lemma allows to derive variational formulas on $\Mgn$ by restricting the already known variational formulas on $\strat$.

Put $z(x)=\int_{x_1}^x v$; this is a local coordinate on $\Ch\setminus \{x_1,\dots,x_{\gh-1}\}$.

\begin{lemma} \la{varBhteo}
Assuming that both $z(x)$ and $z(y)$ are kept fixed, we have
\be
\f{\p B_\pm (x,y)}{\p z_i} = \f{1}{2\pi \sqrt{-1}}\int_{ s_i^*} \f{B_\pm(x,\cdot)\, B_\pm(\cdot\,,y)}{v(\cdot)} 
\la{varBh}
\ee
where the coordinates $z_i$ are given by (\ref{zi}), $s_i^*$ denotes the $i$-th element of the dual basis $\{\b_i^-,-\a_{i}^-\}$,  $i=1,\dots,3g-3$, and the dot stands for the integration variable as before.
\end{lemma}
\begin{proof}
As in Lemma \ref{lemmacoord}, we pick a basis in
$H_1(\Ch;\{x_1,\dots,x_{\gh-1}\})$ given by the classes
$\{a_i,b_i\},\;i=1,\dots,\gh,$ and the paths $l_j$
connecting $x_1$ with $x_j,\;j=2,\dots,\gh-1$ that do not intersect $a_i,\,b_i$.
Consider the dual basis in $H_1(\Ch\setminus \{x_1,\dots,x_{\gh-1}\})$ given by 
 $\{b_i,-a_i,r_j\}$, where $r_j$ is a 
small positively-oriented loop around $x_j$ not intersecting $a_i,\,b_i$. We extend these two dual bases to all pairs $(\Ct,\tilde{v})\in\strat$. 

Let us cut $\Ct$ along the cycles $a_i,\,b_i$ into a fundamental polygon $\widetilde{F}$. The local coordinate $z(x)$ is then a holomorphic function on $\widetilde{F}$ whose differential vanishes precisely at $x_1,\dots,x_{\gh-1}$. Denote by $B_j$ the disk bounded by $r_j$ and put $\widetilde{F}_j=\widetilde{F}\setminus B_j$.

Denote by $\Bt(x,y)$ the Bergman bidifferential on $\Ct\times\Ct$. Then
for $x,y\in \widetilde{F}\setminus\{x_1,\dots,x_{\gh-1}\}$ with $z(x),\,z(y)$ fixed, the following variational formulas for $\Bt(x,y)$ were proven in \cite{JDG} over the moduli space $\strat$:
\begin{align}
&\f{\p\Bt(x,y) }{\p (\int_{a_i} \vt)}
= \f{-1}{2\pi\sqrt{-1}}\int_{b_i}   \f{ \Bt(x,\cdot)\Bt(\cdot\,,y)}{\vt(\cdot)}\;,
\la{bha}\\
&\f{\p\Bt(x,y) }{\p (\int_{b_i} \vt)}
= \f{1}{2\pi\sqrt{-1}}\int_{a_i}  \f{ \Bt(x,\cdot)\Bt(\cdot\,,y)}{\vt(\cdot)}\;,
\la{bhb}\\
&\f{\p\Bt(x,y)}{\p (\int_{l_j} \vt)}
= \f{1}{2\pi\sqrt{-1}}\int_{r_j} \f{ \Bt(x,\cdot)\Bt(\cdot\,,y)}{\vt(\cdot)}\;,\quad x,y\in\widetilde{F}_j\;.
\la{bhli}
\end{align}

Lemma \ref{lemmacoord} claims, in particular, that the  subspace $\Mgn$ in $\strat$ is given by linear equations in homological coordinates.
Restricting the formulas (\ref{bha}), (\ref{bhb}) and (\ref{bhli}) to the space $\Mgn$
we  get the variational formulas (\ref{varBh}) as follows. 

The derivatives of  $\Bh(x,y)$
with respect to $z_{2j-1}=\int_{\a_j^-}v,\;j=1,\dots,g$ can be immediately computed from (\ref{bha}) using that $\a_j^-=\f{1}{2}(a_j-a_j^*)$ and $\int_{\a_j^{-}}v=\int_{a_j}v= -\int_{a_j^*}v$; the same applies to the derivatives with respect to $z_{2j}=\int_{\b_j^-}v$.

The computation of the derivatives with respect to $z_{2i-1+2g}=\int_{\a_{i+g}^-}v$ for $1=1,\dots,2g-3$ is only slightly more involved. 
Note that $\a_{i+g}^-=\at_i$ by (\ref{abm}). Recalling that any integral over $l_j$ 
connecting $x_1$ and $x_j$ is a linear combination of integrals over $\a_{i+g}^-$ and $\b_{i+g}^-$ with half-integer coefficients, we get
\begin{align}
\f{\p B_-(x,y)}{\p z_{i+g}}&=
\f{-1}{2\pi\sqrt{-1}}\int_{\bt_i} \f{ \Bh(x,\cdot)\left(\Bh(\cdot\,,y)-\mu^*_y\Bh(\cdot\,,y)\right)}{v(\cdot)}\nonumber\\ 
&+\f{1}{2\pi\sqrt{-1}}\sum_{j=2}^{\gh-1}\left(\f{\p }{\p z_{i+g}}\int_{l_j}v\right) \int_{r_j} \f{\Bh(x,\cdot)\left(\Bh(\cdot\,,y)-\mu^*_y\Bh(\cdot\,,y)\right)}{v(\cdot)}  \;.
\end{align}
Since the l.h.s. is symmetric in $x$ and $y$, symmetrizing the r.h.s. we get
\begin{align}
\f{\p B_-(x,y)}{\p z_{i+g}}&=
\f{-1}{2\pi\sqrt{-1}}\int_{\bt_i} \f{B_-(x,\cdot)B_-(\cdot\,,y)}{v(\cdot)}\nonumber\\ 
&+\f{1}{2\pi\sqrt{-1}}\sum_{j=2}^{\gh-1}\left(\f{\p }{\p z_{i+g}}\int_{l_j}v\right) \int_{r_j} \f{B_-(x,\cdot)B_-(\cdot\,,y)}{v(\cdot)}  \;,\quad x,\,y\in\bigcap_{j=2}^{\gh-1}\widetilde{F}_j\;.\la{varbh20}
\end{align}
The first term in the r.h.s. of (\ref{varbh20}) 
gives the necessary contribution. To show that the second term vanishes,
we notice that $\mu_y^*B_-(x,y)=-B_-(x,y)$ and $\mu^*v=-v$,
whereas the loop $r_j$ can be chosen invariant under the action of $\mu$. Therefore,
$$ \int_{r_j}  \f{B_-(x,\cdot) B_-(\cdot\,,y)}{v(\cdot)}=0.$$
In the same way (\ref{varBh}) is proved for the derivatives with respect to $z_{2i+2g}=\int_{\b_{i+g}^{-}}v$, $i=1,\dots,2g-3$. A similar computation is valid for $B_+(x,y)$ as well.
\end{proof}

For any pair $(\Ch,v)$ we define a projective connection
$S_v$ as follows. Pick a local parameter $\zeta(x)$ on $\Ch$
and put
\be
S_v=\left(\frac{v'}{v}\right)'-\frac{1}{2}\left(\frac{v'}{v}\right)^2,
\la{defSv}
\ee
where the prime means the derivative with respect to $\zeta(x)$ (in other words, $S_v$ is the Schwarzian derivative of the abelian integral $\int_{x_0}^x v$).
Taking the limit $y\to x$ in (\ref{varBh}) we get the following
\begin{corollary}
For the projective connections $S_{B_\pm}$ we have
\be
\frac{\p}{\p z_i} \left.\left(\f{S_{B_{\pm}}-S_v}{v} \right)(x)\right|_{z(x)={\rm const}}
=\f{3}{2\pi\sqrt{-1}}\int_{s_i^*} \f{B^2_\pm(x,\cdot\,)}{v(x)\,v(\cdot)}\;,\quad i=1,\dots,6g-6.
\la{derSB}
\ee
\end{corollary}

The next corollary gives the variational formulas for the normalized abelian differentials $u_j^\pm$ and the period matrices $\wO_\pm$:
\begin{corollary}
We have
\begin{align}
&\left.\frac{\p u_j^\pm(x)}{\p z_i} \right|_{z(x)={\rm const}}= 
\int_{s_i^*} \f{ u_j^\pm(\cdot) B_\pm(x,\cdot\,)}{v(\cdot)}\;,
\la{varupm}\\
&\left(\frac{\p \wO_\pm}{\p z_i}\right)_{jk} = \int_{ s_i^*}\f{ u_j^\pm u_k^{\pm}}{v}\;.
\la{varsipm}
\end{align}
\end{corollary}
\begin{proof}
The formula (\ref{varupm}) is a corollary of (\ref{varBh}) and (\ref{perBpm}). In turn, (\ref{varsipm}) follows from
(\ref{varupm}) and the definition of the period matrices $\wO_\pm$ (\ref{sigp}),
(\ref{sigm}).
\end{proof}

\section{Tau functions}
\la{secdeftau}

Here we define the necessary tau functions and study their basic properties. Our definition generalizes the
notion of the Bergman tau function in the theory of Frobenius manifolds \cite{Dubrovin}, isomonodromic
deformations \cite{IMRN} and the spectral theory of flat Laplacians on Riemann surfaces \cite{JDG}. Our exposition closely follows the recent papers \cite{Adv, MRL}.

\subsection{Definition of $\tau_\pm$.}
Consider the two meromorphic quadratic differentials $S_{B_\pm}-S_v$, where $S_{B_\pm}$ are the Bergman projective connections
(\ref{defSbpm}) and $S_v$ is given by (\ref{defSv}). 
Take the trivial line bundle on $\Mgn$ and consider two connections
\be
d_\pm=d+\qb_\pm\label{conn}
\ee
where 
\be
\qb_\pm=-\sum_{i=1}^{6g-6} \left(\int_{s_i^*} \phi_\pm\right)d\int_{s_i} v\;,
\la{qbk}
\ee
$\{s_i\}=\{\a_i,\b_i\}$ and $\{s_i^*\}=\{\b_i\, -\a_i\},\; i=1,\dots,3g-3,$ being the dual bases in $H_-$ as above, 
and the meromorphic abelian differentials $\phi_\pm$ are given by
\be
\phi_\pm=-\f{2}{\pi\sqrt{-1}}\,\f{ S_{B_{\pm}}-S_v}{v}\;.
\la{defOk}
\ee
These connections are flat (cf. \cite{JDG} for details).

\begin{definition}
The tau functions $\tau_\pm$ are (locally) covariant constant sections of the trivial line bundle on $\Mgn$ with the connections $d_\pm$, that is,
\be
d_\pm\tau_\pm=0\;.
\la{deftau}
\ee
\end{definition}

As we discussed in the previous section, the moduli space $\Mgn$ can be considered as a subspace of the
moduli space  $\strat$ of genus $\gh=4g-3$ abelian differentials with $\gh-1$ zeros of multiplicity 2.
Following \cite{MRL}, we define a tau function $\taut$ on $\strat$ as follows.
 
For an arbitrary point $(\Ct,\vt)$ in a tubular neighborhood ${\mathcal U}$ of  $\Mgn$ in
 $\strat$ consider two dual bases in $H_1(\Ct, \{x_1,\dots,x_{\gh-1}\})$ and $H_1(\Ct\setminus \{x_1,\dots,x_{\gh-1}\})$ as in Lemma \ref{varBhteo}, and let $\Bt$
be the corresponding Bergman bidifferential on $\Ct\times\Ct$. Define a connection in the trivial line bundle
on $\mathcal{U}$ by 
\be
\widetilde{d}=d+\widetilde{\qb}\;,
\la{defdh}
\ee
where
\be
\widetilde{\qb}=-\sum_{i=1}^{\gh}\left( \left(\int_{a_i} \Oh\right) d \int_{b_i}\vt
-\left(\int_{b_i} \Oh\right) d\int_{a_i} \vt \right)- \sum_{j=2}^{\gh-1}\left(\int_{r_j} \Oh\right) 
d\int_{l_j} \vt 
\la{qbh}
\ee
and
$$
\Oh=-\f{2}{\pi\sqrt{-1}}\f{ S_{\Bt}-S_{\vt}}{\vt}\;.
$$

The Bergman tau function $\taut$ on  the tubular neighborhood ${\mathcal U}$ of $\Mgn$ in
 the moduli space $\strat$ is a covariant constant section of the trivial line bundle on ${\mathcal U}$
with respect to the connection $\widetilde{d}$:
\be
\widetilde{d}\,\taut = 0\;.
\ee

\begin{lemma}
The restriction $\tauh=\taut\left|_{\Mgn}\right.$  is related to the tau functions $\tau_\pm$ on  $\Mgn$ by
\be
 \tauh^2=\tau_+ \tau_-
\la{reltau}
\ee
\end{lemma}
\begin{proof}
According to (\ref{defSbpm}),  we have the relation 
\be
2\Oh= \phi_++\phi_-\;.\label{phi}
\ee
between the meromorphic differentials $\Oh$ and $\phi_\pm$ on $\Ch$. Moreover,  the following integrals vanish
\be
\int_{r_j} \phi_\pm(x)=\int_{r_j} \Oh(x)=0\;,\qquad j=2,\dots,\gh-1\;,
\la{intvanish}
\ee
since $r_j$ are invariant under the action of $\mu$ while $\Oh$ and $\phi_\pm$ change sign.
Therefore, the second sum in the r.h.s. of (\ref{qbh}) vanishes on $\Mgn$. By means of the symplectic transformation   (\ref{mainmat})
the form $\widehat{\qb}=\widetilde{\xi}\left|_{\Mgn}\right.$ can be re-written in terms of the $(\a,\b)$-periods as follows:
\begin{align}
\widehat{\qb}&=\sum_{i=1}^{2g}\left(\left(\int_{\b_i^+}\Oh\right)d\int_{\a_i^+} v
- \left(\int_{\a_i^+}\Oh\right)d\int_{\b_i^+} v\right) \nonumber\\
&+\sum_{j=1}^{3g-3}\left(\left(\int_{\b_j^{-}}\Oh\right)d\int_{\a_j^{-}} v  - 
\left(\int_{\a_j^{-}}\Oh\right)d\int_{\b_j^{-}} v\right)\;.
\end{align}
All $\a^+$- and $\b^+$-periods of $v$ vanish since $\mu^*v=-v$. Thus by (\ref{phi}) we have 
$2\widehat{\qb}=\qb_-+\qb_+$ that proves the lemma.
\end{proof}

\subsection{Explicit formulas.}
To study transformation properties of the tau functions we will need explicit formulas for them 
(cf. \cite{JDG}). We start with $\tauh$ and denote by $\Eh(x,y)$ the prime form on $\Ch$, by $\wO$ -- the matrix of $b$-periods, and by $\thetah(w,\wO),\; w\in \C^{\gh}$ -- the associated Riemann theta-function.
Consider local coordinates on $\Ch$ that we call {\em natural} (or {\em distinguished}) with respect to the differential $v$. 
We take $\zeta(x)=\int_{x_1}^x v$ as a local coordinate on $\Ch\setminus\{x_1,\dots,x_{4g-4}\}$ and choose a local coordinate $\zetah_k$  
near $x_k\in \Ch$ in such a way that $v=d(\zetah_k^3)=3\zetah_k^2\,d\zetah_k$, $k=1,\dots,4g-4$.
 In terms of these coordinates we have 
$\Eh(x,y)=\frac{\Eh(\zeta(x),\zeta(y))}{\sqrt{d\zeta}(x)\sqrt{d\zeta}(y)}$, and we define
\begin{eqnarray*}
\Eh(\zeta,x_k)&=&\lim_{y\rightarrow x_k}\Eh(\zeta(x),\zeta(y))\sqrt{\frac{d\zetah_k}{d\zeta}}(y),\\
\Eh(x_k,x_l)&=&\lim_{\stackrel{\scriptstyle x\rightarrow x_k}{y\rightarrow x_l}}\Eh(\zeta(x),\zeta(y))
\sqrt{\frac{d\zetah_k}{d\zeta}}(x)\sqrt{\frac{d\zetah_l}{d\zeta}}(y)\,.
\end{eqnarray*}
Let $\Ah^x$ be the Abel map on $\Ch$ with the base point $x$, and let
 $\Kh^x=(\Kh^x_1,\dots,\Kh^x_{\gh})$ be the vector of Riemann constants
\be
\Kh^x_i=\frac{1}{2}+\frac{1}{2}\wO_{ii}-\sum_{j\neq i}\int_{a_i}\left(u_i(y)\int_x^y u_j\right)dy
\label{rc}\ee
(as always, we assume that the integration paths do not intersect the chosen cycles $a_i,b_i$ on $C$);
then $\Ah^x((v))+2\Kh^x=\sh \Zh+\Zh'$ for some $\Zh,\Zh'\in\Z^{\gh}$ . 

The tau function $\tauh=\tauh(\Ch,v)$ is explicitly given by the formula (cf. (3.8) of \cite{MRL} with all $m_k=2$):
\be
\tauh(\Ch,v)=\frac{\left( \left.\left(\sum_{i=1}^{\gh} u_i(\zeta)\frac{\partial}{\partial w_i}\right)^{\gh}
\thetah(w;\wO)\right|_{w=\Kh^{\zeta}}\right)^{16}}{e^{4\pi\sqrt{-1}\langle\wO \Zh+4\Kh^{\zeta},\Zh\rangle}\,\Wh(\zeta)^{16}}\cdot
\f{\prod_{k<l}\Eh(x_k,x_l)^{16}}{\;\;\;\prod_k \Eh(\zeta,x_k)^{16(\gh-1)}}\;,
\label{taukk}
\ee
where $\Wh(\zeta)$ is the Wronskian of the normalized abelian differentials 
$u_1,\dots, u_{\gh}$ on $\Ch$.

\smallskip
The tau function $\tau_+$ has a similar form but it is defined in terms of the curve $C$.
The natural local parameters on $C$ near the points $x_k$ are given by squares of the natural local parameters on the cover $\Ch$:
\be
\zeta_k(x)= \zetah_k^2(x)
\ee
 In terms of these coordinates the prime form on $C$ can be written as
$E(x,y)=\frac{E(\zeta(x),\zeta(y))}{\sqrt{d\zeta}(x)\sqrt{d\zeta}(y)}$, and we put
\begin{eqnarray*}
E(\zeta,x_k)&=&\lim_{y\rightarrow x_k}E(\zeta(x),\zeta(y))\sqrt{\frac{d\zeta_k}{d\zeta}}(y),\\
E(x_k,x_l)&=&\lim_{\stackrel{\scriptstyle x\rightarrow x_k}{y\rightarrow x_l}}E(\zeta(x),\zeta(y))
\sqrt{\frac{d\zeta_k}{d\zeta}}(x)\sqrt{\frac{d\zeta_l}{d\zeta}}(y)\,.
\end{eqnarray*}
The basis of normalized abelian differentials on $C$  is given by $(u_1^+,\dots,u_g^+)$. 
We denote by ${\mathcal A}^x$ the  Abel map on $C$ with the base point $x$, and by $K^x=(K^x_1,\dots,K^x_{g})$ -- the vector of Riemann constants; we have $\f{1}{2}{\mathcal A}^x((q))+2K^x=\O_+ Z+Z'$ for some $Z,Z'\in\left(\f{1}{2}\Z\right)^g$.

The tau function $\tau_+$ is explicitly given by the following formula (cf. (3.8) of \cite{MRL} with all $m_k=1/2$\footnote{Rigorously speaking, formula (3.8) of \cite{MRL} applies only for integer $m_k$. However, its proof is valid for half-integer $m_k$ as well, see \cite{JDG}.}):
\be
\tau_+(C, q)= \f{\left(\left.\left(\sum_{i=1}^g u^+_i(\zeta)\frac{\partial}{\partial w_i}\right)^g
\theta(w;\O_+)\right|_{w=K^{\zeta}}\right)^{16}}{e^{4\pi\sqrt{-1}\langle\O_+ Z+4K^{\zeta},Z\rangle}\;W(\zeta)^{16}}
\cdot\frac{\prod_{k<l}E(x_k,x_l)}
{\quad\prod_k E(\zeta,x_k)^{4(g-1)}}\;,
\label{taukk1}
\ee
where $W(\zeta)$ is the Wronskian of the normalized abelian differentials 
$u_1^+,\dots, u_g^+$ on $C$ and $\theta$ is the Riemann theta-function.

Formula (\ref{reltau}) then yields an explicit expression for the tau function $\tau_-$.

\subsection{Transformation properties of $\tau_\pm$.}
The group $\C^*$ of non-zero complex numbers acts by multiplication on quadratic differential $q$ on $C$
and, therefore, on the tau functions $\tau_\pm$.
\begin{theorem} \la{hompro}
Under the action of $\e\in\C^*$ the tau functions $\tau_\pm$ transform like $\tau_\pm(C,\e q)= \e^{p_\pm} \tau_\pm(C, q)$, where
\be
p_+={\f{20}{3}(g-1)}\;,\quad p_-={\f{44}{3}(g-1)} \;.
\la{homcom}
\ee
\end{theorem}
\begin{proof}
First let us prove that
\be
p_{\pm}=2\sum_{j=1}^{4g-4}{\rm Res}_{x_j}\left(\f{ S_{B_{\pm}}-S_v}{v}(x) \int_{x_j}^x v \right)\;.
\la{pkres}
\ee
The homological coordinates $z_i$ given by the periods of the abelian differential $v=q^{1/2}$
transform as $z_i\mapsto \e^{1/2} z_i $ under the action $q\mapsto \e q$. Therefore, we get
$p_\pm$ applying the Euler vector field $\f{1}{2}\sum_{i=1}^{6g-6} z_i\f{\p}{\p z_i}$
to $\log\tau_{\pm}$:
\be
p_\pm=\f{1}{2}\sum_{i=1}^{6g-6} z_i\f{\p\log\tau_\pm}{\p z_i} = \f{1}{2}\sum_{i=1}^{6g-6} \left(\int_{s_i} v\right)  
\left(\int_{s_i^*}\phi_\pm\right)\;.
\la{hom1}
\ee
In terms of $\{\a^{-}_i,\b^{-}_i\}$, we can rewrite (\ref{hom1})  as
\be
p_{\pm}=-\f{1}{2}\sum_{i=1}^{3g-3}\left(\left(\int_{\a^{-}_i} v\right) \left(\int_{\b^{-}_i}\phi_\pm\right)-
\left(\int_{\b^{-}_i} v\right) \left(\int_{\a^{-}_i}\phi_\pm\right)\right)\;.
\la{hom2}
\ee
Since the periods of $v$ over the elements of $H_+$ vanish, we can extend the summation over the basis 
$\{\a^{\pm}_i\,,\b^{\pm}_i\}$ in $H_1(\Ch,\R)=H_+\oplus H_-$.
Applying the symplectic transformation (\ref{mainmat}), we further rewrite (\ref{hom2})
in terms of the canonical basis (\ref{mainbasis}) in $H_1(\Ch,\Z)$, and use the Riemann bilinear relations to obtain
\begin{align}
p_{\pm} &= \sum_{i=1}^{4g-3} \left(- \left(\int_{a_i} v \right) 
\left(\int_{b_i}\phi_\pm\right)+\left(\int_{b_i} v\right) \left(\int_{a_i}\phi_\pm\right)\right)\nonumber\\
&= -2\pi\sqrt{-1} \sum_{j=1}^{4g-4} {\rm Res}_{x_j}\left( \phi_\pm(x) \int_{x_0}^x v\right)
\la{hom5}
\end{align}
for an arbitrary initial point  $x_0$. Moreover, by (\ref{intvanish})
\be\la{resxl1}
{\rm Res}_{x_j}\left(\phi_\pm(x) \int_{x_0}^x v\right) = {\rm Res}_{x_j}\left(\phi_\pm(x) \left(\int_{x_0}^{x_j} v
+ \int_{x_j}^{x} v\right)\right)=
{\rm Res}_{x_j}\left(\phi_\pm(x) \int_{x_l}^{x} v\right)\;,
\ee
which implies (\ref{pkres}).

Let us now compute the residues. Recall that by (\ref{defSbpm}) we have 
$S_{B_\pm}(x)=S_{\Bh}(x)\pm 6 \mu_y^*\Bh(x,y)\left|_{y=x}\right.$. Choose a coordinate $\zeta$
in a neighborhood of $x_j$ such that
$S_{\Bh}=0$. In this coordinate $v=c\,(\zeta^2+O(\zeta^3))d\zeta$ and $\int_{x_j}^x v = c\,\zeta^{3}/3+O(\zeta^4)$ for 
some constant $c\neq 0$. Therefore, near $x_j$ we have
\begin{align*}
&S_v=\left(\f{v'}{v}\right)'-\f{1}{2}\left(\f{v'}{v}\right)^2
=-4\left(\f{1}{\zeta^2}+O\left(\f{1}{\zeta}\right)\right) d\zeta^2,\\
&\f{S_v}{v}\int_{x_l}^x v = -\f{4}{3}\left(\f{1}{\zeta}+O(1)\right) d\zeta\;,
\end{align*}
so that
\be
{\rm Res}_{x_j}\left(\f{S_{\widehat{B}}-S_v}{v}(x)\int_{x_l}^x v\right)= \f{4}{3}\;.
\la{compres1}
\ee
Moreover, we can choose $\zeta$ in such a way that $\zeta(\mu(x))=-\zeta(x)$. Then near $x_j$ we have
$$
\Bh(x,\mu(x))= \f{d\zeta(x) \,d\zeta(\mu(x))}{(\zeta(x)-\zeta(\mu(x)))^2}+\dots = -\f{1}{4}\left( \f{1}{\zeta^2}+O\left(\f{1}{\zeta}\right)\right) d\zeta^2
$$
and 
\be
{\rm Res}_{x_j}\left(\f{\Bh(x,\mu(x))}{v(x)}\int_{x_j}^x v\right)=-\f{1}{12}\;.
\la{compres2}
\ee
Combining  (\ref{defSbpm}) with (\ref{compres1}) and (\ref{compres2}),
 we get (\ref{homcom}).
\end{proof}

Now we want to study the behavior of the tau functions under a change of the canonical homology basis that commutes with the action of the involution $\mu$. Note that $\tau_+$ (resp. $\tau_-$) is uniquely determined by the subspace in $H_+$ (resp. $H_-$) generated by the images of the classes $\a_i^+$ (resp. $\a_j^-$) under the natural homomorphism $H_1(\Ch\setminus\{x_1,\dots,x_{4g-4}\},\R)\to H_1(\Ch,\R)=H_+\oplus H_-$. 

\begin{theorem}\la{thsymptra}
Let $\s$ be a symplectic transformation of $H_1(\Ch,\R)$ commuting with $\mu_*$ that is given by the matrices $\s^\pm$ in the basis $\{\a_i^+,\b_i^+,\a_j^-,\b_j^-\},\; i=1,\dots,g,\; j=1,\dots,3g-3$,  as in (\ref{Qpm}).
Then the tau functions $\tau_\pm$ transform under the action of $\s$ by the formula
\be
 \f{\tau_\pm^{\s}}{\tau_\pm}=  \gamma_\pm(\s_\pm) \;{\rm det} (C_\pm\Omega_\pm + D_\pm)^{48}
\la{transtau3}
\ee
where $\gamma_\pm(\s_\pm)^3 = 1$.
\end{theorem}
\begin{proof}
From the the definition of the tau functions (\ref{deftau}) we get by means of (\ref{transSb}) that
\be
\f{\p}{\p z_i}\log \f{\tau_\pm^{\s}}{\tau_\pm} = 
48 \int_{s_i^*}\f{ \bu_{\pm}^t(x)   (C_\pm\Omega_\pm + D_\pm)^{-1} C_\pm  \bu_\pm(x)}{v(x)}
\la{transtau1}
\ee
Using the variational formulas (\ref{varsipm}) for the matrices $\widehat{\Omega}_\pm$, we
can rewrite (\ref{transtau1}) as follows:
\be
\f{\p}{\p z_i}\log \f{\tau_\pm^{\s}}{\tau_\pm} =  48 \f{\p}{\p z_i}\log
  {\rm det} (C_\pm\Omega_\pm + D_\pm)\;,
\la{transtau2}
\ee
which implies (\ref{transtau3}) for some constants $\gamma_\pm(\s)$ independent of $z_i$.

To show that $\gamma_\pm(\s)^3 = 1$ we use the explicit formula (\ref{taukk1}) for $\tau_+$ or (\ref{taukk}) together with (\ref{reltau}). The local coordinates 
near $x_k$ in these formulas are given by
\be
\zetah_k(x)=\left(\int_{x_k}^x v\right)^{2/3}
\la{locpar2}
\ee
and are uniquely defined up to a 3rd root of unity that implies $\gamma_\pm(\s_\pm)^3=1$.
\end{proof}

The following statement is an immediate consequence of Theorems \ref{hompro} and \ref{thsymptra}.

\begin{theorem}\label{sect}
The tau function $\tau_+$ and  $\tau_-$ are nowhere vanishing holomorphic sections of the line bundles 
$\lambda^{48}\otimes L^{-\f{20(g-1)}{3}}$ and $\lambda_P^{48}\otimes L^{-\f{44(g-1)}{3}}$
on the (projectivized) moduli space $P\Mgn$ respectively, where $L$ is the tautological line bundle on $P\Mgn$.
\end{theorem}

\begin{corollary}
The Hodge and Prym classes in the rational Picard group of $\Mgn$ satisfy the relations
\be
\lambda=\f{5(g-1)}{36}\psi\;,\qquad \lambda_P=\f{11(g-1)}{36}\psi\;,
\la{PTnon-compp}
\ee
where $\psi=c_1(L)$. In particular, $\lambda_P= \f{11}{5}\, \lambda$ on $P\Mgn$.
\end{corollary}

\subsection{Tau functions on $D_{deg}$.}
Here we briefly discuss the tau functions on the open subset $P\Qc_g^{1}\subset D_{deg}$, where $\Qc_g^{1}$ parameterizes
holomorphic quadratic differentials of degeneracy 1 (that is, with one zero $x_{0}$ of multiplicity $2$ and $4g-6$ simple zeros $x_1,\dots,x_{4g-6}$). We will need it later in studying the asymptotic behavior of the tau functions $\tau_\pm$. The exposition closely follows the previous subsections of this section.

As in the case of generic quadratic differentials, each pair $(C,q)\in\Qc_g^{1}$ gives rise to the canonical cover $\Ch$. The curve $\Ch$ is irreducible, has a node at the preimage of $x_0$ and is smooth everywhere else (slightly abusing the notation, we will use the symbol $\Ch$ for both the covering curve and its normalization).  The abelian differential $v=\sqrt{q}$ on $\Ch$ has $4g-6$ double zeros on $\Ch$ at the branch points of the cover. Additionally,  $v$ has two simple zeros at the preimages $x_0',\;x_0''$ of $x_0\in C$ in the normalization of $\Ch$ . The genus of $\Ch$ is then equal to $4g-4$.

The group $H_1(\Ch,\R)$ decomposes into the direct sum $H_+\oplus H_-$ of $\pm 1$-eigenspaces of $\mu_*$, where ${\rm dim} H_+=2g$,   ${\rm dim} H_-=6g-8$.
Like we did before, we pick a basis $\{s_i\}=\{\a_i,\b_{i-3g+4},\; i=1,\dots, 6g-8,$ in $H_-$. 
We also choose a path $l$ connecting the points $x_0'$ and $x_0''$ on $\Ch$ that does not intersect 
the cycles representing the classes $s_i$. A local coordinate system on $\Qc_g^1$ is given by $6g-7$ homological coordinates
$$z_0=\int_{l}v\;,\qquad
z_i=\int_{s_i} v\;,\quad i=1,\dots,6g-8\;.$$

The tau functions $\tau'_\pm$ on the space  $\Qc^1$ are defined similar to (\ref{conn})-(\ref{deftau}). 
Put
\be
\qb'_\pm=  -\left(\int_{r} \phi_\pm\right) d \int_{l} v -\sum_{i=1}^{6g-8} \left(\int_{s_i^*} \phi_\pm\right) d \int_{s_i} v\;,
\ee
where
\be
\phi_\pm=-\f{2}{\pi\sqrt{-1}}\,\f{ S_{B_{\pm}}-S_{v}}{v}\;,
\la{defOk1}
\ee
$\{s_i^*\}$ is the basis in $H_-$ dual to $\{s_i\}$, and $r$ is a small loop around $x_0''$. Consider the connections $d'_\pm=d+\qb'_\pm$ in the trivial line bundle on $\Qc^1_g$, and define $\tau'_\pm$ by the equation $d'_\pm\tau'_\pm=0$.

Like in the case of $\Mgn$, we denote by $\tauh'$ the restriction to $\Qc_g^1$  of the tau function 
on the stratum  $\H_{4g-4}([2^{4g-6}\;1^{2}])$ of the moduli space of abelian differentials (cf. \cite{JDG}), and similarly to (\ref{reltau}) we have $\tauh'^2= \tau'_-\tau'_+ $.

The homogeneity coefficients of the tau functions $\tau'_\pm$ with respect to the action of $\C^*$ are given by

\begin{lemma}
The tau functions $\tau'_\pm$ transform under the action of $\e\in\C^*$ as follows:
\be
\tau'_+(C,\e q)= \e^{\f{20}{3}g-7} \tau'_+(C, q)\;,\qquad
\tau'_-(C,\e q)= \e^{\f{44}{3}g-19} \tau'_-(C, q)\;.
\la{homcom1}
\ee
\end{lemma}
The proof repeats verbatim the proof of Theorem \ref{hompro}, and we omit it here.

\section{The Hodge and Prym classes on $P\Mgno$}
\la{secasymp}
Here we compute the divisors of the tau functions $\tau_\pm$ viewed as holomorphic sections of line bundles on the compactification $P\Mgno$  of the space $P\Mgn$. The
boundary  $P\Mgno\setminus P\Mgn $ consists of $[g/2]+2$ divisors that we defined in Section 1 and denoted by $D_{deg},\;D_0$ and $D_j, \; j=1,\dots,[g/2]$.

\subsection{Asymptotics of $\tau_\pm$ near $D_{deg}$}
The divisor $D_{deg}$ is the closure of the set of isomorphism classes of pairs $(C,q)$ modulo the action of $\C^*$, where $C$ is a smooth curve and $q$ is a holomorphic quadratic differential with at least one non-simple zero. 
We look for the asymptotic behavior of $\tau_\pm(C,q)$ on $\Mgn$ when two simple zeroes of $q$ (say, $x_1$ and $x_2$) coalesce into a zero of order 2. Such a deformation is described by a family 
$(C_t,q_t)\in\Mgn$, where  $(C_t,q_t)\to (C_0,q_0)\in\Mgn^1$ as $t\to 0$.

Assuming that the zeroes $x_1^t$ and $x_2^t$ of $q_t$ are close to each other, consider a small disc
$U_t\subset C_t$ that contains $x_1^t,\; x_2^t$ and no other zeros of $q_t$. Connect 
$x_1^t$ and $x_2^t$ with a path $l_t\subset U_t$. The path $l_t$ lifts to a non-trivial 
cycle $s_t=f_t^{-1}(l_t)$ on the canonical cover $\Ch_t$ such that $\mu_t(s_t)=-s_t$. Therefore, without loss of generality we can assume that $s_t=\a_1^-\in H_-$. Clearly, $s_t$ contracts to the node of $\Ch_0$ when $t\to 0$ so that the homology class $\a_1^-$ vanishes. 

Recall that the first homological coordinate on $\Mgn$ is given by
\be
z_1=\int_{\a_1^-} v=\int_{s_t} v_t\;,
\la{tdeg}
\ee
where $v_t^2=f_t^*(q_t)$, cf. (\ref{zi}). Thus, $z_1$ vanishes on $\Mgn^1$.

\begin{lemma}
The coordinate $t=z_1$ is a transversal to $\Mgn^1$ local coordinate in a tubular neighborhood of $\Mgn^1$ in $\Mgn$.
\end{lemma}
\begin{proof} 
We can choose a coordinate $\zeta$ on $U_t$ in such a way that
$q(\zeta)=(\zeta-\zeta(x_1^t))(\zeta-\zeta(x_2^t))d\zeta^2$.
Then
\begin{align}
z_1=2\int_{x_1^t}^{x_2^t} v=2\int_{\zeta(x_1^t)}^{\zeta(x_2^t)} 
\sqrt{(\zeta-\zeta(x_1^t))(\zeta-\zeta(x_2^t))}d\zeta= 
\text{const} \cdot (\zeta(x_1^t)-\zeta(x_2^t))^2\;.\label{z1}
\end{align}
Since $ (\zeta(x_1^t)-\zeta(x_2^t))^2$ is independent of the labeling of the zeros of $q_t$, it is a coordinate transversal to $\Mgn$, and so is $t=z_1$.
\end{proof}

The asymptotics of $\tau_\pm$ near $\Mgn^1$ follows from
\begin{lemma}
The tau functions $\tau_\pm(C_t,q_t)$ have the following asymptotic behavior as $t\to 0$:
\be
\tau_\pm(C_t,q_t)= t^{\kappa_\pm}\tau_\pm(C_0,q_0) (c_\pm+o(1))\;,
\la{ddeg}
\ee 
where  $\kappa_+= 2/3$ and  $\kappa_-= 26/3$.
\la{aseven2}
\end{lemma}
\begin{proof}
Let us consider a deformation where $t=z_1\to 0$ and $z_2,\dots,z_{6g-6}$ remain fixed.
From the analysis of Section 3 we have $\f{d}{dt}\log\tau_\pm(C_t,q_t)\to 0$ and
$$\f{\tau_\pm(C_t,q_t)}{\tau_\pm(C_0,q_0)}=c_\pm(t)(1+o(t))$$
as $t\to 0$. Under the multiplication $q\mapsto \e q,\; \e\in\C^*,$ the coordinate $t=z_1$ transforms like $t\mapsto \e^{1/2} t$, cf. (\ref{z1}). From (\ref{homcom}) and (\ref{homcom1}) we get that
\be
c_+(\e^{1/2} t)= \e^{1/3} c_+(t),\quad c_-(\e^{1/2} t)=\e^{13/3} c_-(t)\;.
\la{pp0pm0}
\ee
Therefore, $c_+(t)=t^{2/3}$ and $c_-(t)=t^{26/3}$.
\end{proof}

\subsection{Asymptotics of $\tau_\pm$ near $D_0$}
The divisor $D_0$ is the closure of the set of the isomorphism classes of pairs $(C,q)$ modulo the action of $\C^*$, where $C$ is a curve with a non-separating node and $q$ is a regular meromorphic quadratic differential on the normalization of $C$ with poles of order not higher than 2 at the preimages of the node with equal quadratic residues and holomorphic everywhere else.  The dimension of the space of regular meromorphic quadratic differentials on a stable curve of genus $g$ is $3g-3$, so that all the fibers of the projection $p:\Mgno\to\Mb_g$ have the same dimension.

Consider a one parameter family $(C_t,q_t)\in\Mgn$ that is transversal to $D_0$ at some point $(C_0,t_0)\in D_0$. As $t\to 0$, a homologically non-trivial cycle on $C_t$, say $a_1$, contracts to a node of $C_0$, while $q_t$ degenerates to a regular meromorphic differential on $C_0$ with two second order poles, say $x_0,\;y_0$, on the normalization of $C_0$. The canonical cover $\Ch_t$ deforms into an irreducible curve $\Ch_0$ with two nodes represented by two pairs of points $x'_0,\; y'_0$ and $x''_0,\; y''_0$ on the normalization of $\Ch_0$; under the involution $\mu_0: \Ch_0\to\Ch_0$ we have $\mu_0(x'_0)=x''_0,\;\mu_0(y'_0)=y''_0$. The differential $v_t$ on $\Ch_t$ degenerates to a meromorphic abelian differential $v_0$ on $\Ch_0$ with simple poles at the nodes of $\Ch_0$ with residues of the opposite signs.

For the cycles $a_1$ and $b_1$ on $C_t$ we denote their lifts to $\Ch_t$ by $a_1,a_1^*$ and $b_1,b_1^*$ as in (\ref{mainbasis}). According to (\ref{abm}), put
$$
\a_1^{-} =  \f{1}{2}(a_1-\mu_* a_1)\;,\qquad
\b_1^{-} =  b_1-\mu_* b_1
$$
and consider the homological coordinates corresponding to the cycles  $\a_1^{-}$ and  $\b_1^{-}$:
$$
z_1=\int_{\a_1^{-}} v =  \int_{a_1} v_t \;,\qquad
z_2=\int_{\b_1^{-}} v = 2 \int_{b_1} v_t\;.
$$
Similar to \cite{MRL} (cf. also \cite{Fay73}, pp.~50-52) we can take 
\be
t=\exp\left(2\pi\sqrt{-1}\;\cdot\;\f{\int_{b_1} v_t}{\int_{a_1} v_t}\right)=e^{\pi\sqrt{-1}\;z_2/z_1}
\label{loc0}
\ee
as a local coordinate on $\Mgno$ transversal to $D_0$. Moreover, we can 
assume that $z_1$ remains constant and that $\text{Im}\,z_2/z_1\to +\infty$ as $t\to 0$.
  
Let us denote by $\omega_{x,y}$ the abelian differential of the 3rd kind on $\Ch_0$ having simple poles at $x$ and $y$ with residues $+1$ and $-1$ respectively, and normalized to zero $a$-periods. Since $\mu^*v_0=-v_0$, we have 
\be
v_0=\f{z_1}{2\pi\sqrt{-1}}\; (\omega_{x'_0,\,y'_0}-\; \omega_{x''_0,\,y''_0})+ \text{holomorphic terms}\;.\label{a3}
\ee

To find the asymptotics of $\tau_\pm=\tau_\pm(C_t,q_t)$ near $D_0$ we notice first that, as $t\to 0$, the Bergman bidifferential $\Bh_t(x,y)$ on $\Ch_t\times\Ch_t$ tends to the Bergman  
bidifferential on $\Ch_0\times\Ch_0$ uniformly on compact subsets away from the diagonal. From the 
definition of the tau function we get
\be
\f{\p \log \tau_\pm}{\p z_2} =-\f{2}{\pi\sqrt{-1}}\left( \int_{a_1} \f{S_{\Bh_t}-S_{v_t}}{v_t}\pm
6 \int_{a_1} \f{\mu_y^*\Bh_t(x,y)|_{y=x}}{v_t}\right)\;. \label{dt}
\ee
In the limit $t\to 0$ the integrals over the vanishing cycle $a_1$ in (\ref{dt}) reduce to the residues
\be
-4\, {\rm Res}_{\,x'_0}\f{S_{\Bh_0}-S_{v_0}}{v_0} 
\pm 24\, {\rm Res}_{\,x'_0}\f{\mu_y^*\Bh_0(x,y)|_{y=x}}{v_0}\;.
\ee
Since $\mu_y^*\Bh_0(x,y)|_{y=x}$ remains non-singular as $x\to x'_0$, the second term vanishes.
To compute the first term we choose a coordinate $\zeta$ in a neighborhood of $x'_0$ with $\zeta(x'_0)=0$ such that $S_{\Bh_0}=0$. Then by (\ref{a3})
$$
v_0=\f{z_1}{2\pi\sqrt{-1}}\,\left(\f{1}{\zeta}+ O(1)\right)d\zeta\quad\text{as}\;\;\zeta\to 0\;,
$$
so that
$$
S_{v_0}= \left(\f{v_0'}{v_0}\right)' -\f{1}{2}\left(\f{v_0'}{v_0}\right)^2 =
\left(\f{1}{2\zeta^2}+ O(1)\right)\,d\zeta^2
$$
and
$$
\f{S_{v_0}}{v_0}= -\f{\pi\sqrt{-1}}{z_1}\left(\f{1}{\zeta}+ O(1)\right)d\zeta
$$
(here the prime means the derivative with respect to $\zeta$). Thus,
\be
\lim_{z_2\to \infty}\f{\p\log\tau_\pm(C_t,q_t)}{\p z_2}= \f{4\pi\sqrt{-1}}{z_1}\;,
\ee
which, together with (\ref{loc0}) implies that
\be
\tau_\pm(C_t,q_t)= t^{4} (\text{const} + o(1))\quad\text{as}\;\;t\to 0\;. 
\la{d0}
\ee

\subsection{Asymptotics of $\tau_\pm$ near $D_j$}
The divisor $D_0$ is the closure of the set of the isomorphism classes of pairs $(C,q)$ modulo the action of $\C^*$, where $C$ is a curve with a separating node consisting of two irreducible components $C'$ and $C''$ of genus $j$ and $g-j$ respectively, and $q$ is a regular meromorphic quadratic differential on $C$ with poles of order not higher than 2 at the preimages of the node with equal quadratic residues and holomorphic everywhere else. The restrictions of $q$ to $C'$ and $C''$ we will denote by $q'$ and $q''$ respectively. The dimension of the space of meromorphic quadratic differentials with a pole of order not higher than 2 on a curve of genus $j$ is $3j-1$, so that all fibers of the projection $p|_{D_j}: D_j\to \Delta_j\subset\Mb_g$ have the same dimension $3g-3$.

Consider a one parameter family $(C_t,q_t)\in\Mgn$ that is transversal to $D_j$ at some point $(C_0,q_0)\in D_j$. As $t\to 0$, a homologically trivial cycle on $C_t$, say $s$, contracts to a node of $C_0$, while $q_t$ degenerates to a regular meromorphic differential on $C_0$,  where both $q'$ and $q''$ have second order poles at the node with equal quadratic residues. The canonical cover $\Ch_t$ deforms into a reducible curve $\Ch_0$ with two nodes represented by two pairs of points $x'_0,\;y'_0$ and $x''_0,\;y''_0$, where  $x'_0,\,y'_0\in \Ch'_0$ and $x''_0,\,y''_0\in\Ch''_0$; under the involution $\mu_0: \Ch_0\to\Ch_0$ we have $\mu_0(x'_0)=y'_0,\;\mu_0(x''_0)=y''_0$. The differential $v_t$ on $\Ch_t$ degenerates to a pair of meromorphic abelian differentials $v'_0$ on $\Ch'_0$ and $v''_0$ on $\Ch''_0$ with simple poles at the nodes of $\Ch_0$ with residues of the opposite signs.

For a separating cycle $s$ on $C_t$ its lift to $\Ch_t$ consists of two cycles $s_1$ and $s_2$ such that
$\mu(s_1)=-s_2$. Therefore, we can assume that $s_1$=$\at_1$ as in (\ref{mainbasis}), and let $\bt_1$ be the corresponding $b$-cycle. Note that the classes of $\at_1$ and  $\bt_1$ lie in $H_-$. Consider the homological coordinates corresponding to the cycles  $\at_1$ and  $\bt_1$:
$$
z_1=\int_{\at_1} v_t \;,\qquad
z_2=\int_{\bt_1} v_t
$$
(rigorously speaking, we should write $z_{g+1}$ and $z_{g+2}$ as in (\ref{zi}) instead of $z_1$ and $z_2$). Analogously to (\ref{loc0}) we can show that
$t=e^{\pi\sqrt{-1}\;z_2/z_1}$
is a local coordinate on $\Mgno$ transversal to $D_j$. Moreover, we can also
assume that $z_1$ remains constant and that $\text{Im}\,z_2/z_1\to +\infty$ as $t\to 0$.

The rest is very similar to the case of $D_0$.
The derivative of $\tau_\pm=\tau_\pm(C_t,q_t)$ is given by the formula
\be
\lim_{z_2\to\infty}\f{\p\log\tau_\pm}{\p z_2}= -4\;{\rm Res}_{\,x'_0}\f{S_{\Bh_0}-S_{v_0}}{v_0}\;.
\ee
In the coordinate $\zeta$ near $x'_0$ such that $S_{\Bh_0}=0$ we have 
$$\f{S_{v_0}}{v_0}=\f{1}{v_0}\left(\left(\f{v_0'}{v_0}\right)'-\f{1}{2}\left(\f{v_0'}{v_0}\right)^2\right)
=-\f{d\zeta}{4\pi\sqrt{-1} z_1 \zeta}\;.$$
Therefore, 
$$
\lim_{z_2\to\infty}\f{\p\log\tau_\pm}{\p z_2}= -\f{\pi\sqrt{-1}}{z_1}\;{\rm Res}_{\zeta=0}\left(\f{1}{\zeta}+ O(1)\right)d\zeta\;,
$$
and, since $t=e^{\pi\sqrt{-1} z_2/z_1}$, we have 
\be
\tau_\pm = t^{4}(\text{const}+ o(1))\label{dj}
\ee
as $t\to 0$.

\begin{remark}
There is a subtlety for $j=1$ coming from the fact that nodal curves with an elliptic tail have non-trivial automorphism group $\Z/2\Z$, but this issue is taken care of in the definition of the divisor class $\delta_1$.
\end{remark}

\subsection{Proof of Theorem 1 and its consequences}
Combining Theorem \ref{sect} with the asymptotics (\ref{ddeg}), (\ref{d0}) and (\ref{dj}) of $\tau_\pm$ near $D_{deg},\; D_0$ and $D_j$, we get the formulas
\begin{align}
&48\,\l-\f{20}{3}(g-1)\,\psi=\f{2}{3}\, \delta_{deg}+4\sum_{j=0}^{[g/2]}\delta_j\;,
\la{PTMbarp}\\
&48\,\l_{P}-\f{44}{3}(g-1)\,\psi=\f{26}{3}\, \delta_{deg}+4\sum_{j=0}^{[g/2]}\delta_j\;,
\la{PTMbarm}
\end{align}
that immediately imply Theorem 1.

Excluding $\delta_{deg}$ from (\ref{PTMbarp}) and (\ref{PTMbarm}), we get
\begin{corollary}
The Prym class $\l_{P}$ on $P\Mgno$ is related to the Hodge class  $\l$, the tautological class $\psi$ and
the boundary classes $\d_j,\; j=0,\dots,[g/2]$, by the formula
\be
\l_{P}- 13 \l =-\sum_{j=0}^{[g/2]} \delta_j -\f{3}{2}(g-1)\psi\;.
\la{lptl}
\ee
\end{corollary}

The Prym class $\l_{P}$ is closely related to the Mumford class 
$\l_2=c_1(\pi_*\omega_g^2)$ on the moduli space of curves $\Mb_g$ (in other words, $\l_2$ is the
determinant of the vector bundle $p:\Mgno\to\Mb_g$). For each $(C,q)\in\Mgn$ there is
an isomorphism between the space $H^0(C,T^*C^2)$ of holomorphic quadratic differentials on $C$ and the 
space $\Lambda_-$ of Prym differentials on $\Ch$ defined as follows:
\be
u= \frac{f^*w}{v}\in \Lambda_-\;,\qquad w\in H^0(C,T^*C^2)\;,
\la{l2pt}
\ee
where $f:\Ch \to C$ is the canonical projection and $v^2=f^*q$. It easy to check in local coordinates that
$u$ is holomorphic and that $f^*u=-u$. Under the multiplication $q\mapsto \e q$ 
 the differential $u$ transforms like $u\mapsto \e^{-1/2}u$. Since the dimension of $\Lambda_-$ equals $3g-3$, this implies that $\pt= p^*\l_2-\f{3g-3}{2} \psi$. Thus, (\ref{lptl}) yields
\begin{corollary}
We have
\be
p^*\l_2 - 13 p^*\l_1 =-\sum_{j=0}^{[g/2]} \delta_j \;.
\ee
\end{corollary}
This is the famous Mumford's relation from \cite{M} pulled back to $P\Mgno$ from $\Mb_g$ by means of the projection $p:P\Mgno\to\Mb_g$.

{\bf Acknowledgements.} Both authors acknowledge the hospitality of the Max-Planck-Institut f\"ur Mathematik in Bonn and thank A.~Kokotov for valuable discussions. PZ is grateful to SCGP (Stony Brook) and QGM (Aarhus) for support and stimulating working conditions.  We would like to thank the referee for many useful comments and
suggestions.

\bibliographystyle{amsplain}

\begin{thebibliography}{}
\bibitem{D} Dubrovin, B., {\it Geometry of 2D topological field
theories}, in: {\it Integrable systems and quantum groups.} Proceedings, Montecatini
Terme 1993, Lecture Notes in Math. {\bf 1620}, Springer, 120-348 (1996).

\bibitem{Dubrovin} B.~Dubrovin, {\it  Painlev\'e transcendents in two-dimensional topological field theory, 
The Painlev\'e property}, CRM Ser. Math. Phys., Springer, 287-412 (1999).

\bibitem{EKZ} A.~Eskin, M.~Kontsevich, A.~Zorich, {\it Sum of Lyapunov exponents of the Hodge bundle with respect to the Teichmuller geodesic flow}, arXiv:1112.5872. 

\bibitem{Fay73} J.~Fay, {\it Theta-functions on Riemann surfaces},
Lecture Notes in Math.  {\bf 352}, Springer (1973).

\bibitem{KG} G. van der Geer, A.~Kouvidakis, {\it The Hodge bundle on Hurwitz spaces},
Pure Appl. Math. Quart. {\bf 7}, 1297-1308 (2011).

\bibitem{KG1} G. van der Geer, A. Kouvidakis: {\it The class of a Hurwitz divisor on the moduli of curves of
even genus}, arXiv 1005.0969.

\bibitem{HM}J. Harris, D. Mumford: {\it On the Kodaira dimension of the moduli space of curves}, Invent.
Math. {\bf 67}, 23-86 (1982).

\bibitem{IMRN} A.~Kokotov,  D.~Korotkin, {\it Isomonodromic tau function of Hurwitz Frobenius manifolds and its applications}, IMRN,  {\bf 2006},  1-34 (2006).

\bibitem{JDG} A.~Kokotov, D.~Korotkin, {\it Tau functions on spaces of Abelian
differentials and higher genus generalization of Ray-Singer formula},
J. Diff. Geom. {\bf 82}, 35-100 (2009).

\bibitem{Adv} A.~Kokotov, D.~Korotkin, P.~Zograf, {\it Isomonodromic tau
function on the space of admissible covers}, Adv. Math., {\bf 227} no. 1, 586-600 (2011).

\bibitem{KonZorInv} M.~Kontsevich, A.~Zorich, {\it Connected components of the moduli spaces of Abelian differentials with prescribed singularities}, Invent. Math. {\bf 153}, no. 3, 631-678 (2003).

\bibitem{MRL} D.~Korotkin, P.~Zograf, {\it Tau function and moduli of differentials}, 
Math. Res. Lett. {\bf 18}, no.3, 447-458 (2011).

\bibitem{M} D.~Mumford, {\it Stability of projective varieties}, L'Ens. Math., {\bf 23}, 39-110 (1977).
\end{thebibliography}

\end{document}